\documentclass[a4paper,twoside,10pt]{amsart}
\usepackage{amsmath,amsfonts,amssymb,amsthm}
\usepackage{mathrsfs}
\usepackage{graphicx}
\usepackage{color}
\usepackage[a4paper]{geometry}
\usepackage[colorlinks=true,pdfstartview={XYZ null null 1.00},pdfview=FitH,citecolor=blue,urlcolor=customgreen]{hyperref}
\usepackage{enumerate}
\usepackage{caption}
\newtheorem{theorem}{Theorem}[section]
\newtheorem{lemma}[theorem]{Lemma}
\newtheorem{proposition}[theorem]{Proposition}
\numberwithin{equation}{section}
\def\im{\mathrm{i}}
\def\d{\mathrm{d}}
\def\id{\,\mathrm{d}}
\def\e{\mathrm{e}}

\def\Re{\operatorname{Re}}

\def\Im{\operatorname{Im}}
\def\si{\operatorname{si}}
\def\ci{\operatorname{ci}}
\def\ti{\operatorname{ti}}
\def\Ci{\operatorname{Ci}}
\def\f{\operatorname{f}}
\def\g{\operatorname{g}}
\definecolor{customgreen}{rgb}{0.0, 0.5, 0.0}

\author[G. Nemes]{Gerg\H{o} Nemes}
\address{School of Mathematics, Harbin Institute of Technology, Xidazhi Street, Harbin 150001,\newline Heilongjiang, People's Republic of China\medskip}

\email{g.nemes@hit.edu.cn}


\begin{document}

\title[Asymptotics for the generalised trigonometric integral and its zeros]{Asymptotic expansions for the generalised\\ trigonometric integral and its zeros}

\begin{abstract}
In this paper, we investigate the asymptotic properties of the generalised trigonometric integral $\ti(a, z, \alpha)$ and its associated modulus and phase functions for large complex values of $z$. We derive asymptotic expansions for these functions, accompanied by explicit and computable error bounds. For real values of $a$, the function $\ti(a, z, \alpha)$ possesses infinitely many positive real zeros. Assuming $a < 1$, we establish asymptotic expansions for the large zeros, accompanied by precise error estimates. The error bounds for the asymptotics of the phase function and its zeros will be derived by studying the analytic properties of both the phase function and its inverse. Additionally, we demonstrate that for real variables, the derived asymptotic expansions are enveloping, meaning that successive partial sums provide upper and lower bounds for the corresponding functions.
\end{abstract}

\maketitle

\section{Introduction and main results}\label{sec1}

Let $a$ be a complex number and $\alpha$ a real number satisfying $0 \leq \alpha < 1$. We define the generalised trigonometric integral $\ti(a, z, \alpha)$ as
\[
\ti(a,z,\alpha ) = \ci(a,z)\cos (\pi \alpha ) + \si(a,z)\sin (\pi \alpha ).
\]
Here, $\si(a, z)$ and $\ci(a, z)$ denote B\"ohmer's generalised sine and cosine integrals, respectively (see, e.g., \cite{Bohmer1939}, \cite[\S9.10]{Erdelyi1953}, \cite[\href{http://dlmf.nist.gov/8.21}{\S8.21}]{DLMF}). These functions are expressed in terms of the (upper) incomplete gamma function as follows:
\begin{align*}
\si(a, z) & = \frac{1}{2\im} \e^{\frac{\pi}{2}\im a} \Gamma\left(a, z \e^{-\frac{\pi}{2}\im}\right) - \frac{1}{2\im} \e^{-\frac{\pi}{2}\im a} \Gamma\left(a, z \e^{\frac{\pi}{2}\im}\right),
\\ \ci(a, z) & = \frac{1}{2} \e^{\frac{\pi}{2}\im a} \Gamma\left(a, z \e^{-\frac{\pi}{2}\im}\right) + \frac{1}{2} \e^{-\frac{\pi}{2}\im a} \Gamma\left(a, z \e^{\frac{\pi}{2}\im}\right).
\end{align*}
In particular, $\si(a, z)= \ti\left(a,z,\frac{1}{2}\right)$ and $\ci(a, z)=\ti(a,z,0)$. Generally, the functions $\ti(a, z, \alpha)$, $\ci(a, z)$, and $\si(a, z)$ are multivalued with a branch point at $z = 0$. Additionally, they are entire functions of the parameter $a$. For $m\in \mathbb{Z}$,
\begin{align*}
\ti\left( a,z\e^{2\pi m\im} ,\alpha \right) &= \e^{2\pi m\im a} \ti(a,z,\alpha ) + \left( 1 - \e^{2\pi m\im a} \right)\cos \left( \tfrac{\pi}{2}a - \pi \alpha \right)\Gamma (a),
\\ \si\left(a, z \e^{2\pi m \im}\right) &= \e^{2\pi m \im a} \si(a, z) + \left(1 - \e^{2\pi m \im a}\right) \sin\left(\tfrac{\pi}{2}a\right) \Gamma(a),
\\ \ci\left(a, z \e^{2\pi m \im}\right) &= \e^{2\pi m \im a} \ci(a, z) + \left(1 - \e^{2\pi m \im a}\right) \cos\left(\tfrac{\pi}{2}a\right) \Gamma(a),
\end{align*}
provided $a\notin \mathbb{Z}_{\ge 0}$ (cf. \cite[\href{http://dlmf.nist.gov/8.2.ii}{\S8.2(ii)}]{DLMF}). When $a$ is zero or a negative integer, these formulae remain valid if the right-hand side is replaced by its limiting value.

When $\Re(a)<1$, the following integral representations are valid:
\begin{align*}
\ti(a,z,\alpha) &= \int_z^\infty t^{a - 1} \cos( t+\pi\alpha)\id t,
\\ \si(a,z) &= \int_z^\infty t^{a - 1} \sin t\id t,
\\ \ci(a,z) & = \int_z^\infty t^{a - 1} \cos t\id t,
\end{align*}
which provide justification for the naming of these functions (see \cite[\href{http://dlmf.nist.gov/8.21.iii}{\S8.21(iii)}]{DLMF}). The classical sine and cosine integrals, $\si(z)$ and $\Ci(z)$, appear as special cases: $\si(z) =  - \si(0,z)$ and $\Ci(z) =  - \ci(0,z)$ \cite[\href{http://dlmf.nist.gov/8.21.v}{\S8.21(v)}]{DLMF}. These functions have applications in fields such as electromagnetic theory \cite{Lebedev1965}, optics \cite{Iizuka2008}, and signal processing \cite{Manolakis2011}.

The aim of this paper is to derive asymptotic expansions for the function $\ti(a, z, \alpha)$, along with its modulus and phase functions (defined below), as $z$ becomes large. For these expansions, we will provide explicit and computable error bounds. For real values of $a$, the function $\ti(a, z, \alpha)$ has infinitely many positive real zeros (see Section \ref{sec5}). Under the assumption that $a < 1$, we establish an asymptotic expansion for the large zeros, together with corresponding error bounds. The error bounds for the asymptotics of the phase function and the zeros will be derived by studying the analytic properties of the phase function and its inverse. Furthermore, we will demonstrate that for real variables, our asymptotic expansions are enveloping, meaning that successive partial sums provide upper and lower bounds for the corresponding functions.

To present our results, we introduce some additional notation. We define the auxiliary functions $\f(a, z)$ and $\g(a, z)$ as follows:
\begin{equation}\label{fdef}
\f(a, z) = \frac{1}{2\im} \e^{\frac{\pi}{2}\im a} \e^{-\im z} \Gamma\left(a, z \e^{-\frac{\pi}{2}\im}\right) - \frac{1}{2\im} \e^{-\frac{\pi}{2}\im a} \e^{\im z} \Gamma\left(a, z \e^{\frac{\pi}{2}\im}\right)
\end{equation}
and
\begin{equation}\label{gdef}
\g(a, z) = \frac{1}{2} \e^{\frac{\pi}{2}\im a} \e^{-\im z} \Gamma\left(a, z \e^{-\frac{\pi}{2}\im}\right) + \frac{1}{2} \e^{-\frac{\pi}{2}\im a} \e^{\im z} \Gamma\left(a, z \e^{\frac{\pi}{2}\im}\right).
\end{equation}
Similar to the generalised trigonometric integral, these functions are multivalued in general, with a branch point at $z=0$, and they are entire in the parameter $a$. Using this notation, the generalised trigonometric integral can be written as
\begin{equation}\label{tfg}
\ti(a,z,\alpha ) =  - \f(a,z)\sin (z- \pi \alpha ) + \g(a,z)\cos (z - \pi \alpha ).
\end{equation}
In particular, the generalised sine and cosine integrals can be expressed as
\begin{align}
\si(a,z) & = \f(a,z)\cos z + \g(a,z)\sin z ,\label{sfg}\\ \ci(a,z) & =  - \f(a,z)\sin z + \g(a,z)\cos z.\label{cfg}
\end{align}
The advantage of these auxiliary functions is that they exhibit non-oscillatory behaviour as $z$ becomes large. A useful property of these functions is given by
\begin{equation}\label{fgdiff}
\frac{\d \f(a,z)}{\d z}  = - \g(a,z), \quad \frac{\d \g(a,z)}{\d z}  = \f(a,z) - z^{a-1}.
\end{equation}
For $\Re(a) < 1$ and $|\arg z|< \pi$, these results can be easily derived from the known trigonometric integral representations of these functions \cite[Eqs. \href{http://dlmf.nist.gov/8.21.E22}{(8.21.22)} and \href{http://dlmf.nist.gov/8.21.E23}{(8.21.23)}]{DLMF}, and the restrictions on $a$ and $z$ can subsequently be lifted through analytic continuation. For real values of $a$ and $z$, with $a < 1$ and $z > 0$, we further define  
\begin{equation}\label{Mphidef}
M(a,z)\e^{\im\varphi (a,z)}  = \e^{\frac{\pi}{2}\im a} \Gamma \left( a,z\e^{ - \frac{\pi}{2}\im} \right),
\end{equation}
where the modulus function $M(a,z)$ and the phase function $\varphi (a,z)$ are continuous real-valued functions of $a$ and $z$. The branch of $\varphi(a, z)$ is determined by
\[
\lim _{z \to 0^ +} \varphi (a,z) = \tfrac{\pi}{4}(a + \left| a \right|),
\]
as suggested by the local behaviour of $\Gamma(a, z)$ near $z = 0$ (see \cite[Eqs. \href{http://dlmf.nist.gov/8.4.E15}{(8.4.15)} and \href{http://dlmf.nist.gov/8.7.E3}{(8.7.3)}]{DLMF}). From the integral representations \eqref{fgrepr} and the identity \eqref{M2}, it follows that $M(a,z)>0$ for $z > 0$, ensuring that both $M(a,z)$ and $\varphi(a,z)$ are well defined. With these definitions, we have
\begin{equation}\label{timodulusphase}
\ti(a,z,\alpha ) = M(a,z)\cos (\varphi (a,z) - \pi \alpha)
\end{equation}
for $a<1$ and $z > 0$. In particular, the generalised sine and cosine integrals can be written as
\[
\si(a,z) = M(a,z)\sin (\varphi (a,z)), \quad \ci(a,z) = M(a,z)\cos (\varphi (a,z)).
\]
We also observe that
\begin{equation}\label{M2}
M^2(a, z) =\si^2(a, z) + \ci^2(a, z) =\f^2(a, z) + \g^2(a, z).
\end{equation}
Finally, for $\Re(p)>0$, we define
\begin{equation}\label{Pidef}
\Pi_p(z)=z^p \f(1-p,z).
\end{equation}
The function $\Pi_p(z)$ was originally introduced by Dingle \cite[pp. 407]{Dingle1973}, and following his convention, we refer to it as a basic terminant (although Dingle's notation slightly differs from ours, e.g., $\Pi_{p-1}(z)$ is used for our $\Pi_p(z)$).

We now present the main results of the paper, starting with the asymptotic expansions of $\f(a,z)$ and $\g(a,z)$, which are well-documented in the literature:
\begin{equation}\label{fasymp}
\f(a,z) \sim z^{a - 1} \sum_{n = 0}^\infty  ( - 1)^n \frac{\left(1 - a\right)_{2n} }{z^{2n}}
\end{equation}
and
\begin{equation}\label{gasymp}
\g(a,z) \sim z^{a - 1} \sum_{n = 0}^\infty  ( - 1)^n \frac{\left(1 - a\right)_{2n + 1} }{z^{2n + 1}} ,
\end{equation}
as $z\to\infty$ in the sector $|\arg z|\le \pi- \delta$ ($<\pi$), uniformly with respect to bounded complex values of $a$ (see, for instance, \cite[\S9.10]{Erdelyi1953} or \cite[\href{http://dlmf.nist.gov/8.21.viii}{\S8.21(viii)}]{DLMF}). Asymptotic expansions for the generalised trigonometric integrals can be derived by combining these results with \eqref{tfg}, \eqref{sfg}, and \eqref{cfg}. In the following proposition, we provide explicit expressions for the remainder terms of the asymptotic expansions \eqref{fasymp} and \eqref{gasymp}.

\begin{proposition}\label{prop1} For each non-negative integer $N$, the auxiliary functions $\f(a,z)$ and $\g(a,z)$ can be expressed as
\begin{equation}\label{fexp}
\f(a, z) = z^{a - 1} \left( \sum_{n = 0}^{N - 1} (-1)^n \frac{\left(1 - a\right)_{2n}}{z^{2n}} + R_N^{(\f)}(a,z) \right),
\end{equation}
and
\begin{equation}\label{gexp}
\g(a, z) = z^{a - 1} \left( \sum_{n = 0}^{N - 1} (-1)^n \frac{\left(1 - a\right)_{2n + 1}}{z^{2n + 1}} + R_N^{(\g)}(a,z) \right),
\end{equation}
where the remainder terms are given by the following expressions in terms of the basic terminant:
\begin{equation}\label{Rfgexpr}
R_N^{(\f)}(a,z)=(-1)^N \frac{\left(1 - a\right)_{2N}}{z^{2N}} \Pi_{2N + 1 - a}(z),\quad R_N^{(\g)}(a,z) = (-1)^N \frac{\left(1 - a\right)_{2N + 1}}{z^{2N + 1}} \Pi_{2N + 2 - a}(z),
\end{equation}
provided $2N + 1 > \Re(a)$ for $R_N^{(\f)}(a,z)$ and $2N + 2 > \Re(a)$ for $R_N^{(\g)}(a,z)$.
\end{proposition}

By combining Proposition \ref{prop1} with the explicit bounds for the basic terminant provided in Appendix \ref{terminant}, we obtain explicit and computable error estimates for the asymptotic expansions of the auxiliary functions $\f(a, z)$ and $\g(a, z)$. Specifically, Proposition \ref{propt1} shows that when $z > 0$, with $2N + 1 > a$ for $\f(a, z)$ and $2N + 2 > a$ for $\g(a, z)$, the remainder terms $R_N^{(\f)}(a, z)$ and $R_N^{(\g)}(a, z)$ are bounded in magnitude by the corresponding first neglected term and share its sign. This constitutes the enveloping property of the expansions \eqref{fasymp} and \eqref{gasymp}.

We will now discuss the asymptotic expansion of the function $M^2(a, z)$ and its associated error bounds. The identity \eqref{M2} enables us to extend $M^2(a,z)$ as an analytic function of $z$ in the domain $|\arg z| < \pi$ and as an entire function of $a$. Moreover, combining \eqref{M2} with \eqref{fgdiff}, it can be readily demonstrated that
\begin{equation}\label{Mdiff}
\frac{\d}{\d z}M^2 (a,z) =  - 2 z^{a - 1} \g(a,z).
\end{equation}
Hence, by \eqref{gasymp} and the well-known theorem on the integration of asymptotic power series \cite[Lemma 1.2, \S1.5]{Temme2015}, the asymptotic expansion of $M^2(a, z)$ is immediately obtained:
\begin{equation}\label{Masymp}
M^2(a, z) \sim z^{2a - 1} \sum_{n=0}^\infty (-1)^n \frac{\left(1 - a\right)_{2n+1}}{n + 1 - a} \frac{1}{z^{2n+1}} ,
\end{equation}
as $z\to\infty$ in the sector $|\arg z|\le \pi- \delta$ ($<\pi$), uniformly with respect to bounded complex values of $a$. The fact that the constant of integration is zero follows from \eqref{M2}, \eqref{fasymp}, and \eqref{gasymp}. Note that the apparent singularity in the coefficients at positive integer values of $a$ is, in fact, removable. The following theorem provides bounds for the remainders of the asymptotic expansion \eqref{Masymp}. These error bounds can be further simplified by using the estimates for $\sup_{r \geq 1} \left| \Pi_{p}(zr) \right|$ presented in Appendix \ref{terminant}.

\begin{theorem}\label{thm1} For each non-negative integer $N$, the square of the modulus function can be written as
\begin{equation}\label{Mexp}
M^2(a, z) = z^{2a - 1} \left(\sum_{n=0}^{N-1} (-1)^n \frac{\left(1 - a\right)_{2n+1}}{n + 1 - a} \frac{1}{z^{2n+1}} + R_N^{(M)}(a, z) \right),
\end{equation}
where, for $|\arg z| < \pi$ and $N+1 > \Re(a)$, the remainder term $R_N^{(M)}(a, z)$ is related to the remainder term $R_N^{(\g)}(a, z)$ by
\begin{equation}\label{RMexpr}
R_N^{(M)} (a,z) = 2\int_1^{+\infty} s^{2a - 2} R_N^{(\g)} (a,zs)\id s,
\end{equation}
and satisfies the following bound:
\begin{equation}\label{RMbound}
\left| R_N^{(M)}(a, z) \right| \le \frac{\left| (1 - a)_{2N + 1} \right|}{N + 1 - \Re(a)} \frac{1}{\left| z \right|^{2N + 1}} \sup_{r \geq 1} \left| \Pi_{2N + 2 - a}(zr) \right|.
\end{equation}
Furthermore, when $z > 0$ and $N+1 > a$, the remainder term $R_N^{(M)}(a, z)$ does not exceed the first omitted term in absolute value and is of the same sign.
\end{theorem}

We note that an asymptotic expansion for the modulus function $M(a,z)$ can be derived by taking the square root of the expansion \eqref{Masymp} and expanding in powers of $z$. However, obtaining realistic error bounds for the resulting expansion that are valid over large domains of $z$ and $a$ appears to be challenging.

We now proceed with the analysis of the asymptotic expansion of the phase function $\varphi(a, z)$ and its associated error bounds. As shown in Appendix \ref{coefficients}, the derivative of $\varphi(a, z)$ is given by
\begin{equation}\label{phidiff}
\frac{\d \varphi(a,z)}{\d z} = z^{a - 1}\frac{\f(a,z)}{ M^2(a,z)}.
\end{equation}
Using the asymptotic form of $M^2(a,z)$ in \eqref{Masymp}, it follows that for any $a<1$, there exists a non-negative constant $C_a$ such that $\varphi(a,z)$ extends analytically to the domain $\left\{ z:\left|\arg z\right| < \pi ,\left| z \right| > C_a \right\}$. By \eqref{fasymp}, \eqref{Masymp}, and the well-known theorems on the ratios and integration of asymptotic power series \cite[Lemmas 1.1 and 1.2, \S1.5]{Temme2015}, the asymptotic expansion of $\varphi(a,z)$ is given by
\begin{equation}\label{phiasymp}
\varphi(a,z) \sim z + \frac{\pi}{2} - \sum_{n=0}^\infty (-1)^n \frac{t_n(1-a)}{2n+1}\frac{1}{z^{2n + 1}}
\end{equation}
as $z\to\infty$ in the sector $\left|\arg z\right|\le \pi- \delta$ ($<\pi$), uniformly with respect to bounded real values of $a$ with $a<1$. As shown in Appendix \ref{coefficients}, the coefficients $t_n(1-a)$ are monic polynomials in $1-a$ of degree $2n+1$ with integer coefficients, which can be computed using a recurrence relation. When determining the constant of integration $\frac{\pi}{2}$ in \eqref{phiasymp}, we used the asymptotics
\[
\e^{\frac{\pi }{2}\im a} \Gamma \left( a,z\e^{ - \frac{\pi}{2}\im} \right) = z^{a-1}\e^{\left( z + \frac{\pi }{2} \right)\im} (1 + o(1)), \quad z\to+\infty
\]
(see, e.g., \cite[\href{http://dlmf.nist.gov/8.11.i}{\S8.11(i)}]{DLMF}) along with the continuity of the phase function $\varphi(a,z)$. Combining \eqref{Masymp}, \eqref{phiasymp}, and \eqref{timodulusphase} gives an alternative asymptotic expansion for $\ti(a,z,\alpha)$. The following theorem provides bounds for the remainders of the asymptotic expansion \eqref{phiasymp} when $\left|\arg z\right|<\frac{\pi}{2}$.

\begin{theorem}\label{thm2} The phase function $\varphi(a, z)$ can be analytically extended to the right half-plane $\Re(z) > 0$. For any non-negative integer $N$, this extended function admits the following expansion:
\begin{equation}\label{phiexpansion}
\varphi(a,z) = z + \frac{\pi}{2} - \sum_{n=0}^{N-1} (-1)^n \frac{t_n(1-a)}{2n+1}\frac{1}{z^{2n + 1}} +R_N^{(\varphi)}(a,z),
\end{equation}
where the remainder term $R_N^{(\varphi)}(a, z)$ satisfies the bound
\[
\left| R_N^{(\varphi)}(a, z)\right| \le \frac{t_N (1 - a)}{2N + 1}\frac{1}{\left| z \right|^{2N + 1} } \times \begin{cases} 1 & \text{if }\; \left|\arg z\right| \leq \frac{\pi}{4}, \\ \left|\csc ( 2\arg z)\right| & \text{if }\; \frac{\pi}{4} < \left|\arg z\right| < \frac{\pi}{2}. \end{cases}
\]
Additionally, if $z > 0$, the remainder term $R_N^{(\varphi)}(a, z)$ does not exceed the first neglected term in absolute value and has the same sign. The coefficients $t_n(1 - a)$ are monic polynomials in $1 - a$ of degree $2n + 1$ with integer coefficients, and they satisfy $t_n(1 - a) > 0$.
\end{theorem}

We now turn our attention to the zeros of the function $\ti(a, z, \alpha)$. In this discussion, we continue to assume that $a$ is a real number satisfying $a < 1$. From the integral representations \eqref{fgrepr} and the identity \eqref{M2}, it follows that $\f(a,z)$ and $M^2(a,z)$ are positive for $z > 0$. Consequently, by \eqref{phidiff}, the phase function $\varphi(a, z)$ is strictly monotonically increasing with respect to $z$ along the positive real axis for any fixed $a$. Furthermore, since 
\[
\lim_{z \to 0^ +} \varphi (a,z) = \tfrac{\pi}{4}(a+\left|a\right|) \;\text{ and }\; \varphi(a,z) \sim z \;\text{ as }\; z\to +\infty,
\]
$\varphi(a,z)$ defines an injective mapping from $(0, +\infty)$ to $\left(\frac{\pi}{4}(a+\left|a\right|), +\infty\right)$ for each fixed $a$. Therefore, the representation \eqref{timodulusphase} implies that $\ti(a,z,\alpha)$ has infinitely many positive real $z$-zeros when $a < 1$. (The case $a = 1$ is trivial, as $ \ti(1, z, \alpha) = -\sin(z - \pi \alpha)$, so we exclude it from our discussion. For the case $a>1$, refer to Section \ref{sec5}.) Let $z^\ast$ denote a positive zero of $\ti(a,z,\alpha)$. From \eqref{timodulusphase}, we have
\[
\varphi(a,z^\ast)=\pi\left(k+\alpha+\tfrac{1}{2}\right)
\]
for some integer $k$. Since the range of $\varphi(a,z)$ consists of real numbers greater than $\frac{\pi}{4}(a+\left|a\right|)$, the inequality
\begin{equation}\label{kcond}
k+\alpha >\tfrac{1}{4}(a + \left| a \right| - 2)
\end{equation}
must hold. Because $\varphi(a,z)$ is injective, for any integer $k$ satisfying \eqref{kcond}, there is a unique pre-image of $\pi\left(k + \alpha + \frac{1}{2}\right)$ under $\varphi(a,z)$. This preimage is a positive root of $\ti(a,z,\alpha)$. Furthermore, all positive roots can be obtained in this way. We denote the $k^{\text{th}}$ positive zero of $\ti(a,z,\alpha)$, arranged in ascending order, by $z_\kappa(a)$, where $\kappa = k + \alpha > \frac{1}{4}(a + \left| a \right| - 2)$. Accordingly, the indexing begins at $k = -1$ or $k = 0$, depending on whether $\frac{1}{4}(a + \left| a \right| - 2) - \alpha < -1$ or $\frac{1}{4}(a + \left| a \right| - 2) - \alpha \geq -1$, respectively. In particular, using the notation from \cite[\href{http://dlmf.nist.gov/6.13}{\S6.13}]{DLMF}, the zeros of the classical sine and cosine integrals are given by $s_k = z_{k+1/2}(0)$ and $c_k = z_k(0)$ for $k = 0, 1, 2, \ldots\,$. It follows that
\[
\varphi (a,z_\kappa  (a)) -\tfrac{\pi}{2}= \pi( k + \alpha )  = \pi\kappa.
\]
Since $\varphi(a, z) - \frac{\pi}{2}$ is an injective function of $z$, it has an inverse function, $X(a, w)$, say. For real $a < 1$ and $w > \frac{\pi}{4}(a+\left|a\right|-2)$, $X(a, w)$ is a continuous, strictly monotonically increasing, positive real-valued function of $a$ and $w$ that satisfies
\begin{equation}\label{Xrelation}
X(a, \pi\kappa ) = z_\kappa  (a).
\end{equation}
Thus, to derive an asymptotic expansion for the large zeros (i.e., as $\kappa \to +\infty$, or equivalently $k \to +\infty$), it suffices to obtain a large-$w$ asymptotic expansion for $X(a, w)$. Applying the results from \cite[Theorems 2.3 and 2.4]{Olver1999} on the reversion of asymptotic expansions, we find that for any $a < 1$, there exists a non-negative constant $K_a$ such that $X(a, w)$ extends holomorphically to any closed annular sector properly contained within the region $\{w : |\arg w| < \pi, |w| > K_a\}$. Moreover, noting that the expansion \eqref{phiasymp} for $\varphi(a, z)$ includes only odd powers of $z$, the function $X(a, w)$ admits the asymptotic expansion
\begin{equation}\label{Xasymp}
X(a,w) \sim w + \sum_{n = 0}^\infty ( - 1)^n \frac{c_n(1-a)}{2n+1}\frac{1}{w^{2n+1}}
\end{equation}
as $w\to\infty$ in the sector $\left|\arg w\right|\le \pi- \delta$ ($<\pi$), uniformly with respect to bounded real values of $a$ with $a<1$. As shown in Appendix \ref{coefficients}, the coefficients $c_n(1-a)$ are monic polynomials in $1-a$ of degree $2n+1$, with rational coefficients that can be computed via a recurrence relation. From \eqref{Xrelation} and \eqref{Xasymp}, it follows that the large zeros of the generalised trigonometric integral $\ti(a, z, \alpha)$ have the asymptotic expansion
\begin{equation}\label{zasymp}
z_\kappa(a) \sim \pi\kappa + \sum_{n = 0}^\infty (-1)^n \frac{c_n(1-a)}{2n+1}\frac{1}{(\pi\kappa)^{2n+1}}
\end{equation}
as $\kappa \to +\infty$ (or equivalently $k \to +\infty$), uniformly for bounded real values of $a$ with $a < 1$. The following theorem provides bounds for the remainders of the asymptotic expansion \eqref{Xasymp} when $\left|\arg w\right|<\frac{\pi}{2}$.

\begin{theorem}\label{thm3} The function $X(a, w)$ can be analytically extended to the right half-plane $\Re(w) > 0$. For any non-negative integer $N$, this extended function admits the following expansion:
\begin{equation}\label{Xtrunc}
X(a,w)=w + \sum_{n = 0}^{N-1} ( - 1)^n \frac{c_n(1-a)}{2n+1}\frac{1}{w^{2n+1}} +R_N^{(X)}(a,w),
\end{equation}
where the remainder term $R_N^{(X)}(a, w)$ satisfies the bound
\[
\left| R_N^{(X)}(a, w)\right| \le \frac{c_N (1 - a)}{2N + 1}\frac{1}{\left| w \right|^{2N + 1} } \times \begin{cases} 1 & \text{if }\; \left|\arg w\right| \leq \frac{\pi}{4}, \\ \left|\csc ( 2\arg w)\right| & \text{if }\; \frac{\pi}{4} < \left|\arg w\right| < \frac{\pi}{2}. \end{cases}
\]
Additionally, if $w > 0$, the remainder term $R_N^{(X)}(a, w)$ does not exceed the first neglected term in absolute value and has the same sign. The coefficients $c_n(1 - a)$ are monic polynomials in $1 - a$ of degree $2n + 1$ with rational coefficients, and they satisfy $c_n(1 - a) > 0$.
\end{theorem}

Setting $w = \pi \kappa > 0$ in this theorem yields error bounds for the asymptotic expansion \eqref{zasymp} of the zeros $z_\kappa(a)$. In the language of enveloping series, Theorem \ref{thm3} asserts that the asymptotic expansion \eqref{zasymp} envelopes the zeros $z_\kappa(a)$ in the strict sense, meaning that successive partial sums serve as alternating upper and lower bounds.

We note that the paper \cite{LSVA2016} also provides bounds for the zeros of the generalised sine and cosine integrals $\si(a, z)$ and $\ci(a, z)$ for $a<1$. Specifically, the authors established the following bounds:
\[
\pi \kappa < z_\kappa(a) < \pi \kappa + \frac{c_0 (1 - a)}{\pi \kappa}  \cfrac{2}{1 + \sqrt{1 + \cfrac{4(1 - a)}{\left(\pi \kappa\right)^2}}},
\]
where $\kappa = k + \frac{1}{2}$ in the case of $\si(a, z)$ and $\kappa = k$ in the case of $\ci(a, z)$. These bounds can be compared with those given in Theorem \ref{thm3}.

The remainder of the paper is organised as follows. In Section \ref{sec2}, we establish Proposition \ref{prop1} and Theorem \ref{thm1}, which address the remainder terms in the asymptotic expansions of the auxiliary functions $\f(a,z)$, $\g(a,z)$, and the square of the modulus function $M(a,z)$. Sections \ref{sec3} and \ref{sec4} are dedicated to the proofs of Theorems \ref{thm2} and \ref{thm3}, which provide bounds on the remainder terms of the asymptotic expansions of the phase function $\varphi(a,z)$ and the inverse function $X(a,w)$, respectively. The paper concludes with a discussion in Section \ref{sec5}.

\section{Proof of Proposition \ref{prop1} and Theorem \ref{thm1}}\label{sec2}

In this section, we prove Proposition \ref{prop1} and Theorem \ref{thm1} concerning the remainder terms in the asymptotic expansions of the auxiliary functions $\f(a,z)$ and $\g(a,z)$, and the square of the modulus function $M(a,z)$.

\begin{proof}[Proof of Proposition \ref{prop1}] Our starting point is the integral representation
\begin{equation}\label{Gammaint}
\Gamma(a,z) = \frac{z^{a-1} \e^{-z}}{\Gamma(1-a)} \int_0^{+\infty} \frac{s^{-a} \e^{-s}}{1 + s/z} \id s
\end{equation}
which is valid for $|\arg z| < \pi$ and $\Re(a) < 1$ \cite[Eq. \href{http://dlmf.nist.gov/8.6.E4}{(8.6.4)}]{DLMF}. By applying this formula to \eqref{fdef} and \eqref{gdef}, we obtain the following representations:
\begin{equation}\label{fgrepr}
\f(a, z) = \frac{z^{a - 1}}{\Gamma(1 - a)} \int_0^{+\infty} \frac{s^{-a} \e^{-s}}{1 + \left(s/z\right)^2} \id s, \quad 
\g(a, z) = \frac{z^{a - 1}}{\Gamma(1 - a)} \frac{1}{z} \int_0^{+\infty} \frac{s^{1 - a} \e^{-s}}{1 + \left(s/z\right)^2} \id s,
\end{equation}
where $\left|\arg z\right| < \frac{\pi}{2}$, $\Re(a) < 1$ for $\f(a,z)$, and $\Re(a) < 2$ for $\g(a,z)$. Substituting the identity 
\begin{equation}\label{geom}
\frac{1}{1 + \left(s/z\right)^2 } = \sum_{n = 0}^{N - 1} \frac{( - 1)^n}{z^{2n} }s^{2N}   + \frac{( - 1)^N }{z^{2N} }\frac{s^{2N} }{1 + \left(s/z\right)^2 }
\end{equation}
into the integral representations and integrating term-by-term yields \eqref{fexp} and \eqref{gexp}, with the remainder terms given by
\begin{equation}\label{Rfeq}
R_N^{(\f)}(a,z) = (-1)^N \frac{(1 - a)_{2N}}{z^{2N}} \frac{1}{\Gamma(2N + 1 - a)} \int_0^{+\infty} \frac{s^{2N - a} \e^{-s}}{1 + \left(s/z\right)^2} \id s
\end{equation}
and
\begin{equation}\label{Rgeq}
R_N^{(\g)}(a,z) = (-1)^N \frac{(1 - a)_{2N + 1}}{z^{2N + 1}} \frac{1}{\Gamma(2N + 2 - a)} \int_0^{+\infty} \frac{s^{2N + 1 - a} \e^{-s}}{1 + \left(s/z\right)^2} \id s,
\end{equation}
respectively. By analytic continuation in $a$, \eqref{Rfeq} holds for $2N+1>\Re(a)$, while \eqref{Rgeq} holds for $2N+2>\Re(a)$. Finally, from \eqref{Pidef} and \eqref{fgrepr}, we have
 \[
\Pi _p (z) = \frac{1}{\Gamma (p)}\int_0^{ + \infty } \frac{s^{p - 1} \e^{ - s} }{1 + \left(s/z\right)^2 }\id s 
\]
where $\left|\arg z\right| < \frac{\pi}{2}$ and $\Re(p) > 0$. Using this in \eqref{Rfeq} and \eqref{Rgeq}, and lifting the restriction on $\arg z$ through analytic continuation, completes the proof.
\end{proof}

\begin{proof}[Proof of Theorem \ref{thm1}]
For each non-negative integer $N$, complex $a$, and for $|\arg z| < \pi$, we define the remainder $R_N^{(M)}(a,z)$ via \eqref{Mexp}. From \eqref{Masymp}, we have
\begin{equation}\label{RMorder}
R_N^{(M)}(a,z)=\mathcal{O}\left( z^{-2N-1}\right)
\end{equation}
as $z\to\infty$ in the sector $|\arg z|\le \pi- \delta$ ($<\pi$), uniformly with respect to bounded complex values of $a$. Differentiating \eqref{Mexp} with respect to $z$, we obtain
\begin{equation}\label{Mdiffexp}
\frac{\d}{\d z} M^2(a, z) = -2 z^{2a - 2} \left( 
\sum_{n = 0}^{N - 1} (-1)^n \frac{(1 - a)_{2n + 1}}{z^{2n + 1}} 
- \frac{(2a - 1)}{2} R_N^{(M)}(a, z) 
- \frac{z}{2} \frac{\d}{\d z} R_N^{(M)}(a, z) 
\right).
\end{equation}
Substituting \eqref{gexp} into the right-hand side of \eqref{Mdiff} and comparing with \eqref{Mdiffexp} gives the following first-order inhomogeneous equation for $R_N^{(M)}(a, z)$:
\[
\frac{2a - 1}{2} R_N^{(M)}(a, z) + \frac{z}{2} \frac{\d}{\d z} R_N^{(M)}(a, z) + R_N^{(\g)}(a, z) = 0.
\]
From \eqref{gasymp}, we obtain
\[
R_N^{(\g)}(a,z)=\mathcal{O}\left( z^{-2N-1}\right)
\]
as $z \to \infty$ in the sector $|\arg z| \leq \pi - \delta$ ($< \pi$), uniformly with respect to bounded complex values of $a$. Considering this along with \eqref{RMorder}, we deduce that
\[
R_N^{(M)} (a,z) =\frac{2}{z^{2a - 1} }\int_z^\infty  t^{2a - 2} R_N^{(\g)} (a,t)\id t= 2\int_1^{+\infty} s^{2a - 2} R_N^{(\g)} (a,zs)\id s ,
\]
provided $|\arg z| < \pi$ and $N+1 > \Re(a)$. Thus, from \eqref{Rfgexpr}, we obtain
\begin{align*}
\left| R_N^{(M)}(a, z) \right| 
& \le 2 \int_1^{+\infty} s^{2 \Re(a) - 2} \left| R_N^{(\g)}(a, zs) \right| \id s \\ & =2 \frac{\left| (1 - a)_{2N + 1} \right|}{|z|^{2N + 1}} \int_1^{+\infty} s^{2 \Re(a) - 2N - 3} \left| \Pi_{2N + 2 - a}(zs) \right| \id s \\ &
\leq 2 \frac{\left| (1 - a)_{2N + 1} \right|}{|z|^{2N + 1}} \int_1^{+\infty} s^{2 \Re(a) - 2N - 3} \id s \times \sup_{r \ge 1} \left| \Pi_{2N + 2 - a}(zr) \right|.
\\ & = \frac{\left| (1 - a)_{2N + 1} \right|}{N + 1 - \Re(a)} \frac{1}{\left| z \right|^{2N + 1}} \sup_{r \geq 1} \left| \Pi_{2N + 2 - a}(zr) \right|,
\end{align*}
which completes the proof of the bound \eqref{RMbound}.

Formula \eqref{Rfgexpr} for $R_N^{(\g)}(a,z)$ and Proposition \ref{propt1} imply that
\[
0 < ( - 1)^N R_N^{(\g)} (a,z) < \frac{(1 - a)_{2N + 1} }{z^{2N + 1}},
\]
provided $z>0$ and $2N+2>a$. Substituting this into \eqref{RMexpr}, we immediately obtain
\[
0 < ( - 1)^N R_N^{(M)} (a,z) < \frac{(1 - a)_{2N + 1} }{N + 1 - a}\frac{1}{z^{2N + 1} },
\]
under the conditions that $z > 0$ and $N+1 > a$. Therefore, when $z > 0$ and $N + 1 > a$, the remainder term $R_N^{(M)}(a,z)$ does not exceed the first omitted term in absolute value and is of the same sign.
\end{proof}

\section{Proof of Theorem \ref{thm2}}\label{sec3}
In this section, we prove Theorem \ref{thm2}, starting with the statement and proof of two lemmas.

\begin{lemma}\label{lemma1} The phase function $\varphi(a, z)$ can be analytically extended to the right half-plane $\Re(z) > 0$, and it can be continuously extended to its boundary, $\Re (z) = 0$.
\end{lemma}

\begin{proof} Suppose $z > 0$. Using the definition of $\varphi(a, z)$ and the Schwarz reflection principle, we obtain the following representation:
\begin{equation}\label{phiexpr}
\e^{2\im \varphi(a, z)} = \frac{M(a, z) \e^{\im \varphi(a, z)}}{M(a, z) \e^{-\im  \varphi(a, z)}} = \frac{\e^{\frac{\pi}{2} \im a} \Gamma \left( a, z \e^{-\frac{\pi}{2} \im } \right)}{\overline{\e^{\frac{\pi}{2} \im  a} \Gamma \left( a, z \e^{-\frac{\pi}{2} \im } \right)}} = \e^{\pi \im a} \frac{\Gamma \left( a, z \e^{-\frac{\pi}{2} \im } \right)}{\Gamma \left( a, z \e^{\frac{\pi}{2} \im } \right)}.
\end{equation}
From \eqref{Gammaint}, we deduce that $\Gamma(a, z) \neq 0$ for $a < 1$ and $|\arg z| < \pi$. Therefore, the rightmost fraction in \eqref{phiexpr} has a holomorphic logarithm in the (simply connected) domain $\Re(z) > 0$. Thus,
\begin{equation}\label{phiexpr2}
\varphi(a,z) = \frac{1}{2\im} \log\left(\e^{\pi \im a} \frac{\Gamma\left(a, z\e^{- \frac{\pi}{2}\im}\right)}{\Gamma\left(a, z\e^{\frac{\pi}{2}\im}\right)}\right)
\end{equation}
yields the required holomorphic extension of the phase function $\varphi(a,z)$ to the right
half-plane $\Re (z) > 0$. Here, $\log$ denotes the principal branch of the logarithm. The branch has been chosen to match the limiting value of $\varphi(a, z)$ as $z \to 0$ (cf. \eqref{philimit} below). Furthermore, by \cite[Eq. \href{http://dlmf.nist.gov/8.2.E10}{(8.2.10)}]{DLMF},
\[
\Im\left(\e^{ \mp \pi \im a} \Gamma \left( a,s\e^{ \pm \pi \im}  \right) \right) =  \mp \frac{\pi}{\Gamma (1 - a)} \neq 0
\]
for any $a < 1$ and $s > 0$. Thus, \eqref{phiexpr2} defines a continuous function of $z$ for $a < 1$, $\Re(z) \geq 0$, and $z \neq 0$. To verify the continuity at the origin, we use the known behaviour of $\Gamma(a, z)$ near $z = 0$ (see \cite[Eqs. \href{http://dlmf.nist.gov/8.4.E15}{(8.4.15)} and \href{http://dlmf.nist.gov/8.7.E3}{(8.7.3)}]{DLMF}), which yields
\begin{equation}\label{philimit}
\lim_{z\to 0^+}\frac{1}{2\im} \log\left(\e^{\pi \im a} \frac{\Gamma\left(a, z\e^{- \frac{\pi}{2}\im}\right)}{\Gamma\left(a, z\e^{\frac{\pi}{2}\im}\right)}\right) =\frac{1}{2\im} \log \left( \e^{\pi \im a} \e^{\frac{\pi}{2} \im \left( |a| - a \right)} \right) = \tfrac{\pi}{4} \left( a + |a| \right),
\end{equation}
which is consistent with the definition of $\varphi(a, z)$.
\end{proof}

At this point, it is convenient to introduce the function $\Phi(a, z)$ defined by
\begin{equation}\label{Phidef}
\Phi (a,z)=\varphi(a,z)-\frac{\pi}{2}.
\end{equation}
Clearly, this function inherits the continuity and analyticity properties of $\varphi(a, z)$ as stated in Lemma \ref{lemma1}. In particular, we have
\begin{equation}\label{Phidef2}
\Phi(a,z) = \frac{1}{2\im} \log\left(\e^{\pi \im (a-1)} \frac{\Gamma\left(a, z\e^{- \frac{\pi}{2}\im}\right)}{\Gamma\left(a, z\e^{\frac{\pi}{2}\im}\right)}\right)
\end{equation}
for $|\arg z| \leq \frac{\pi}{2}$ and $a < 1$.

\begin{lemma}\label{lemma2} The function $\Phi(a, z)$ possesses the following properties:

\begin{enumerate}[(i)]\setlength{\itemsep}{5pt}
\item $\Phi(a, z) = z + \mathcal{O}\left(z^{-1}\right)$ as $z \to \infty$ in the closed sector $|\arg z| \leq \frac{\pi}{2}$.
\item $\Re \Phi\left(a, s \e^{\frac{\pi}{2}\im}\right) = o\left(s^{-r}\right)$ as $s \to +\infty$ for any $r > 0$.
\item $\Re \Phi\left(a, s \e^{\frac{\pi}{2}\im}\right) < 0$ for any $s > 0$.
\end{enumerate}
\end{lemma}

\begin{proof} (i) follows immediately from the asymptotic expansion \eqref{phiasymp} and the definition \eqref{Phidef}.

From \eqref{phiasymp} and \eqref{Phidef}, for any $N \geq 0$, we obtain the following:
\begin{align*}
  \Re \Phi\left(a, s \e^{\frac{\pi}{2}\im}\right)& = \Re \left( \Phi\left(a, s \e^{\frac{\pi}{2}\im}\right)-\im s + \sum_{n=0}^{N-1} (-1)^n \frac{t_n(1 - a)}{2n + 1} \frac{1}{(\im s)^{2n + 1}} \right) \\ &= \Re \mathcal{O}\left( \frac{1}{s^{2N+1}} \right) = \mathcal{O}\left( \frac{1}{s^{2N+1}} \right)  
\end{align*}
as $s \to +\infty$, confirming claim (ii).

To prove claim (iii), we proceed as follows. From \eqref{Phidef2}, we have
\begin{equation}\label{Phiim}
\Phi\left(a, s\e^{\frac{\pi}{2}\im}\right) = -\frac{1}{2\im} \log\left( \e^{-\pi \im(a - 1)} \frac{\Gamma\left(a, s\e^{\pi \im}\right)}{\Gamma(a, s)} \right)
\end{equation}
for $s>0$. From \cite[Eq. \href{http://dlmf.nist.gov/8.2.E11}{(8.2.11)}]{DLMF}, we know that
\begin{equation}\label{Gammaexpr}
\e^{-\pi \im (a - 1)} \Gamma\left(a, s\e^{\pi \im}\right) = \Gamma(a)\left(-\cos(\pi a) + s^a \gamma^\ast(a, -s)\right) + \frac{\pi \im}{\Gamma(1 - a)},
\end{equation}
where
\[
\gamma ^\ast (a,z) = z^{ - a} \left( 1 - \frac{\Gamma (a,z)}{\Gamma (a)} \right)
\]
is an entire function of both $a$ and $z$, which is real when both $a$ and $z$ are real. When $a = -n$ for some $n \in \mathbb{Z}_{\geq 0}$, the right-hand side of \eqref{Gammaexpr} is interpreted as its limiting value:
\[
\frac{1}{n!} \left( \log(s\e^{\pi \im}) - \psi(n+1) \right) + 
\sum_{\substack{k=0 \\ k \neq n}}^{\infty} 
\frac{s^{k-n}}{k!(k-n)}
\]
(cf. \cite[Eq. \href{http://dlmf.nist.gov/8.4.E15}{(8.4.15)}]{DLMF}). Consequently, we have
\[
\Re \Phi\left(a, s\e^{\frac{\pi}{2}\im}\right) = -\frac{1}{2} \arg\left( \frac{\Gamma(a)\left(-\cos(\pi a) + s^a \gamma^\ast(a, -s)\right)}{\Gamma(a, s)} + \frac{\pi \im}{\Gamma(1 - a)\Gamma(a, s)} \right).
\]
Claim (iii) follows immediately, since the term $\frac{\pi}{\Gamma(1 - a) \Gamma(a, s)}$ is positive for $a < 1$ and $s > 0$.
\end{proof}

\begin{figure}[ht!]
    \centering
    \includegraphics[width=0.25\textwidth]{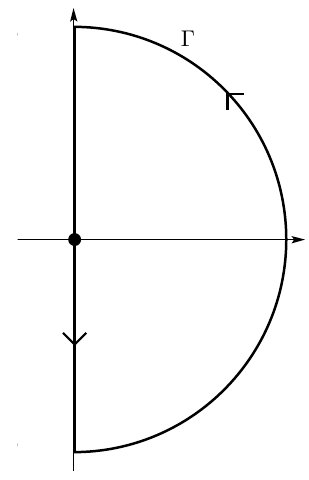} 
    \caption{The integration contour $\Gamma$.}
    \label{figure1}
\end{figure}

\begin{proof}[Proof of Theorem \ref{thm2}] Let $\Gamma$ be the positively oriented, $D$-shaped closed contour depicted in Figure \ref{figure1}. By Cauchy's integral theorem in its most general form \cite[Chap. 2, \S 2.3, Theorem 2.4]{Kodaira2007}, we have the following representation:
\[
\varphi (a, z) - z - \frac{\pi}{2} = \Phi(a,z) - z = \frac{1}{2\pi \im} \oint_\Gamma \frac{\Phi(a,t) - t}{t - z} \id t - \underbrace{\frac{1}{2\pi \im} \oint_\Gamma \frac{\Phi(a,t) - t}{t + z} \id t}_0
\]
for all $z > 0$ lying inside the contour $\Gamma$. From part (i) of Lemma \ref{lemma2}, it follows that as the radius of the semicircular portion of the contour $\Gamma$ tends to infinity, the integrals along it approach zero. Therefore, we have
\begin{align*}
\varphi(a, z) - z - \frac{\pi}{2} = \; & 
\frac{1}{2\pi \im} \int_{+\im\infty}^0 \frac{\Phi(a,t) - t}{t - z} \id t 
+ \frac{1}{2\pi \im} \int_0^{-\im\infty} \frac{\Phi(a,t) - t}{t - z} \id t \\
& - \frac{1}{2\pi \im} \int_{+\im\infty}^0 \frac{\Phi(a,t) - t}{t + z} \id t 
- \frac{1}{2\pi \im} \int_0^{-\im\infty} \frac{\Phi(a,t) - t}{t + z} \id t \\
= \; & - \frac{1}{2\pi} \int_0^{+\infty} \frac{\Phi\left(a,s\e^{\frac{\pi}{2} \im}\right) - \im s}{\im s - z} \id s 
+ \frac{1}{2\pi} \int_0^{+\infty} \frac{\Phi\left(a,s\e^{-\frac{\pi}{2} \im}\right) + \im s}{\im s + z} \id s \\
& + \frac{1}{2\pi} \int_0^{+\infty} \frac{\Phi\left(a,s\e^{\frac{\pi}{2} \im}\right) - \im s}{\im s + z} \id s 
- \frac{1}{2\pi} \int_0^{+\infty} \frac{\Phi\left(a,s\e^{-\frac{\pi}{2} \im}\right) + \im s}{\im s - z} \id s \\
= \; & \frac{1}{z}\frac{2}{\pi} \int_0^{+\infty} \frac{\Re \Phi\left(a,s\e^{\frac{\pi}{2} \im}\right)}{1 + \left(s/z\right)^2} \id s.
\end{align*}
In the final step, we used the fact that
\[
2\Re \Phi\left(a,s\e^{\frac{\pi}{2} \im}\right)=\Phi\left(a,s\e^{\frac{\pi}{2} \im}\right)+\Phi\left(a,s\e^{-\frac{\pi}{2} \im}\right)
\]
for all $s > 0$, which follows from the Schwarz reflection principle. The restriction on $z$ can now be removed via analytic continuation. Thus, we obtain
\begin{equation}\label{phiint}
\varphi(a, z)= z+ \frac{\pi}{2} +\frac{1}{z}\frac{2}{\pi} \int_0^{+\infty} \frac{\Re \Phi\left(a,s\e^{\frac{\pi}{2} \im}\right)}{1 + \left(s/z\right)^2} \id s
\end{equation}
provided that $\Re(z)>0$. Note that, by part (ii) of Lemma \ref{lemma2}, $\Re \Phi\left(a,s\e^{\frac{\pi}{2} \im}\right)$ decays faster than any negative power of $s$ as $s \to +\infty$. Therefore, for any positive integer $N$, with $\Re(z) > 0$ and $s > 0$, we can expand the integrand in \eqref{phiint} using \eqref{geom} and term-by-term integration, to deduce the truncated expansion \eqref{phiexpansion} with
\begin{equation}\label{tcoeffint}
t_n(1 - a) = - \frac{4n + 2}{\pi} \int_0^{+\infty} s^{2n} \Re \Phi\left(a,s\e^{\frac{\pi}{2} \im}\right) \id s
\end{equation}
and
\begin{equation}\label{Rphiint}
R_N^{(\varphi)}(a, z) = (-1)^N \frac{1}{z^{2N + 1}} \frac{2}{\pi} \int_0^{+\infty} \frac{s^{2N} \Re \Phi\left(a,s\e^{\frac{\pi}{2} \im}\right)}{1 + \left(s/z\right)^2} \id s.
\end{equation}
Formula \eqref{tcoeffint} and part (iii) of Lemma \ref{lemma2} together imply that $t_n(1 - a) > 0$. The fact that $t_n(1 - a)$ is a monic polynomial in $1 - a$ of degree $2n + 1$ with integer coefficients is shown in Appendix \ref{coefficients}.

It remains to prove the claimed bounds for the remainder term $R_N^{(\varphi)}(a, z)$. To do so, we combine the elementary inequality
\[
\left| 1 + u^2 \right| \geq \begin{cases} 1 & \text{if }\; \left|\arg u\right| \leq \frac{\pi}{4}, \\ |\sin (2\arg u)| & \text{if }\; \frac{\pi}{4} < \left|\arg u\right| < \frac{\pi}{2}, \end{cases}
\]
with the fact that $\Re \Phi\left(a,s\e^{\frac{\pi}{2} \im}\right)$ does not change sign for $s > 0$ (as stated in part (iii) of Lemma \ref{lemma2}), and use formula \eqref{tcoeffint} (with $N$ in place of $n$) to deduce
\begin{align*}
\left| R_N^{(\varphi)}(a, z) \right| &\le \frac{1}{\left| z \right|^{2N + 1}} \frac{2}{\pi} \int_0^{+\infty} \frac{s^{2N} \left| \Re \Phi\left(a,s\e^{\frac{\pi}{2} \im}\right) \right|}{\left| 1 + \left( s/z \right)^2 \right|} \id s \\
&= \frac{1}{\left| z \right|^{2N + 1}} \frac{2}{\pi} \left| \int_0^{+\infty} \frac{s^{2N} \Re \Phi\left(a,s\e^{\frac{\pi}{2} \im}\right)}{\left| 1 + \left(s/z \right)^2 \right|} \id s \right| \\
&\le \frac{1}{\left| z \right|^{2N + 1}} \frac{2}{\pi} \left| \int_0^{+\infty} s^{2N} \Re \Phi\left(a,s\e^{\frac{\pi}{2} \im}\right) \id s \right| \times \begin{cases} 
1 & \text{if }\; |\arg z| \leq \frac{\pi}{4}, \\
\left| \csc(2 \arg z) \right| & \text{if }\; \frac{\pi}{4} < |\arg z| < \frac{\pi}{2}
\end{cases} \\
&= \frac{t_N(1 - a)}{2N + 1} \frac{1}{\left| z \right|^{2N + 1}} \times \begin{cases} 
1 & \text{if }\; |\arg z| \leq \frac{\pi}{4}, \\
\left| \csc(2 \arg z) \right| & \text{if }\; \frac{\pi}{4} < |\arg z| < \frac{\pi}{2}.
\end{cases}
\end{align*}
It follows directly from part (iii) of Lemma \ref{lemma2}, \eqref{tcoeffint}, and \eqref{Rphiint} that, for $z > 0$, the remainder term $R_N^{(\varphi)}(a, z)$ and $(-1)^{N+1} \frac{t_N(1-a)}{2N+1} \frac{1}{z^{2N + 1}}$ have the same sign.
\end{proof}

\section{Proof of Theorem \ref{thm3}}\label{sec4}
In this section, we establish Theorem \ref{thm3}. We begin by stating and proving two lemmas.

\begin{lemma}\label{lemma3} The function $w = \Phi(a, z)$, with $\Re(z) > 0$, is injective, and its range includes the closed right half-plane $\Re(w) \geq 0$.
\end{lemma}

\begin{proof} From \eqref{Phiim}, we can deduce that
\[
\Im \Phi\left(a, s\e^{\frac{\pi}{2}\im}\right) = \frac{1}{4} \log\left( \left( \frac{\Gamma(a)\left(-\cos(\pi a) + s^a \gamma^\ast(a, -s)\right)}{\Gamma(a, s)} \right)^2 + \left( \frac{\pi}{\Gamma(1 - a)\Gamma(a, s)} \right)^2 \right).
\]
for all $s>0$. Next,
\[
\frac{\d}{\d s} \left(\Gamma(a)\left(-\cos(\pi a) + s^a \gamma^\ast(a, -s)\right)\right) = s^{a-1} \e^s
\;\text{ and }\;
\frac{\d \Gamma (a,s)}{\d s} =  - s^{a - 1} \e^{ - s} .
\]
Thus, for each fixed $a$, $\Im \Phi\left(a, s \e^{\frac{\pi}{2} \im}\right)$ is a strictly monotonically increasing function of $s$. Since $\lim_{s \to 0^+} \Im \Phi\left(a, s \e^{\frac{\pi}{2} \im}\right) = 0$, we conclude that $\Im \Phi\left(a, s \e^{\frac{\pi}{2} \im}\right) > 0$ for all $s > 0$. By Lemma \ref{lemma2}, we also have $\Re \Phi\left(a, s\e^{\frac{\pi}{2}\im}\right) < 0$ for $s > 0$. Applying the Schwarz reflection principle, we obtain
\[
\Phi\left(a, s\e^{-\frac{\pi}{2}\im}\right) =\Re \Phi\left(a, s\e^{\frac{\pi}{2}\im}\right)-\im \Im \Phi\left(a, s\e^{\frac{\pi}{2}\im}\right)
\]
for any $s>0$. Therefore, the function $\Phi(a, z)$ is injective on the imaginary axis, mapping it onto an infinite curve $\mathcal{G}_a$ that lies entirely in the left half-plane and is symmetric with respect to the real axis. The curve $\mathcal{G}_a$ crosses the negative real axis at the point
\[
\lim_{s \to 0^+} \Phi\left(a, s\e^{\pm \frac{\pi}{2}\im}\right) = \lim_{s \to 0^+} \varphi\left(a, s\e^{\pm \frac{\pi}{2}\im}\right) - \frac{\pi}{2} = \tfrac{\pi}{4}(a + |a| - 2).
\]
To continue, we require that $\Phi(a, z)$ is not only continuous but also analytic on the imaginary axis, except at the origin. To establish this, observe that
\[
\Gamma (a,z) = \e^{ - z} U(1 - a,1 - a,z),
\]
where the confluent hypergeometric function $U$ has finitely many zeros in the sector $|\arg z| < \frac{5\pi}{4}$ \cite[\href{http://dlmf.nist.gov/13.9.ii}{\S13.9(ii)}]{DLMF}. Thus, for any fixed $a < 1$, the ratio
\[
\frac{\Gamma\left(a, z\e^{- \frac{\pi}{2}\im}\right)}{\Gamma\left(a, z\e^{\frac{\pi}{2}\im}\right)}
\]
is well-defined and non-zero near $\arg z = \pm \frac{\pi}{2}$. Consequently, \eqref{phiexpr2} defines a holomorphic function of $z$ in a neighbourhood of the rays $\arg z = \pm \frac{\pi}{2}$. The mapping
\[
\zeta  = \zeta (z) = \frac{1}{1 + z}
\]
defines a bijection between the imaginary axis and the circle
\[
\mathcal{C}=\left\{ \zeta :\left| \zeta  - \tfrac{1}{2} \right| = \tfrac{1}{2} \right\},
\]
as well as between the open right half-plane and the interior $I(\mathcal{C})$ of $\mathcal{C}$. Define the function
\[
\Phi^\ast (a,\zeta ) = \Phi\left(a, \tfrac{1}{\zeta } - 1  \right)
\]
on the closure of $I(\mathcal{C})$. Note that $\Phi^\ast(a, \zeta)$ is analytic on the closure of $I(\mathcal{C})$, except at the points $\zeta = 0$ and $\zeta = 1$. It is continuous at $\zeta = 1$, and at $\zeta = 0$, it is of infinitely large order 1:
\[
\lim_{\zeta  \to 0} \zeta \Phi^\ast  (a,\zeta ) = 1
\]
(this follows from the asymptotic behaviour of $\Phi(a, z)$). Since $\Phi(a, z)$ is injective on the imaginary axis, $\Phi^\ast(a, \zeta)$ is injective on $\mathcal{C}$. Thus, by a general version of the boundary correspondence principle \cite[Chap. 4, \S18, Theorem 4.9]{Markushevich1965}, $\Phi^\ast(a, \zeta)$ is injective on the interior of $\mathcal{C}$. The image of the interior is a domain with boundary $\mathcal{G}_a$, which corresponds to the region on the left of an observer moving along the path $w = \Phi^\ast(a, \zeta)$ as $\zeta$ traverses $\mathcal{C}$ once in the positive direction. By the properties of $\mathcal{G}_a$, this domain includes the closed half-plane $\Re(w) \geq 0$ (see Figure \ref{figure2}).
\end{proof}

\begin{figure}[!ht]
    \centering
    \includegraphics[width=0.5\textwidth]{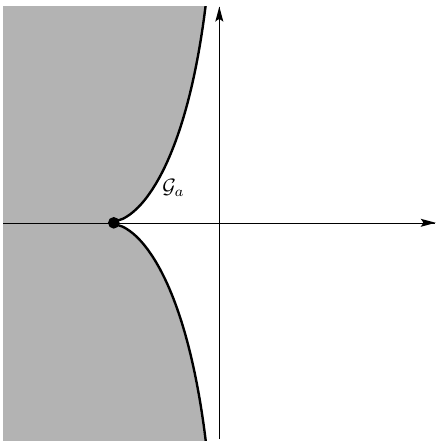} 
    \caption{The range of the function $w = \Phi(a,z)$ (unshaded), bounded by the curve $\mathcal{G}_a$, which intersects the real axis at $w = \frac{\pi}{4}(a + |a| - 2)$.}
    \label{figure2}
\end{figure}

Since $\Phi(a,z)$ is holomorphic and injective in the domain $\Re(z) > 0$, it has a holomorphic inverse defined on its range. This inverse must be identical to the function $X(a,w)$ introduced in Section \ref{sec1}. Therefore, the function $z = X(a,w)$ extends analytically to a domain that includes the closed half-plane $\Re(w) \geq 0$, with its range being the region $\Re(z) > 0$.

\begin{lemma}\label{lemma4} The function $X(a, w)$ possesses the following properties:

\begin{enumerate}[(i)]\setlength{\itemsep}{5pt}
\item $X(a, w) = w + \mathcal{O}\left(w^{-1}\right)$ as $w \to \infty$ in the closed sector $|\arg w| \leq \frac{\pi}{2}$.
\item $\Re X\left(a, s \e^{\frac{\pi}{2}\im}\right) = o\left(s^{-r}\right)$ as $s \to +\infty$ for any $r > 0$.
\item $\Re X\left(a, s \e^{\frac{\pi}{2}\im}\right) > 0$ for any $s > 0$.
\end{enumerate}
\end{lemma}

\begin{proof} (i) follows immediately from the asymptotic expansion \eqref{Xasymp}.

From \eqref{Xasymp}, for any $N \geq 0$, we obtain the following:
\begin{align*}
  \Re X\left(a, s \e^{\frac{\pi}{2}\im}\right)& = \Re \left( X\left(a, s \e^{\frac{\pi}{2}\im}\right)-\im s - \sum_{n=0}^{N-1} (-1)^n \frac{c_n(1 - a)}{2n + 1} \frac{1}{(\im s)^{2n + 1}} \right) \\ &= \Re \mathcal{O}\left( \frac{1}{s^{2N+1}} \right) = \mathcal{O}\left( \frac{1}{s^{2N+1}} \right)  
\end{align*}
as $s \to +\infty$, confirming claim (ii).

Since the domain of the function $X(a,w)$ includes the imaginary axis and its range is the right half-plane, claim (iii) follows immediately.
\end{proof}

\begin{proof}[Proof of Theorem \ref{thm3}] The proof follows in exactly the same way as the proof of Theorem \ref{thm2}, so we only present the resulting integral formulas, which may be of independent interest. By following the initial steps in the proof of Theorem \ref{thm2} and referring to Lemma \ref{lemma4}, we obtain the expression
\[
X(a, w)= w+ \frac{1}{w}\frac{2}{\pi} \int_0^{+\infty} \frac{\Re X\left(a,s\e^{\frac{\pi}{2} \im}\right)}{1 + \left(s/w\right)^2} \id s
\]
for $\Re(w)>0$. Expanding the integrand using \eqref{geom} (with $w$ in place of $z$), we recover the truncated expansion \eqref{Xtrunc}, with the coefficients given by
\begin{equation}\label{ccoeffint}
c_n(1 - a) = \frac{4n + 2}{\pi} \int_0^{+\infty} s^{2n} \Re X\left(a,s\e^{\frac{\pi}{2} \im}\right) \id s
\end{equation}
and the remainder term
\begin{equation}\label{Xint}
R_N^{(X)}(a, w) = (-1)^N \frac{1}{w^{2N + 1}} \frac{2}{\pi} \int_0^{+\infty} \frac{s^{2N} \Re X\left(a,s\e^{\frac{\pi}{2} \im}\right)}{1 + \left(s/w\right)^2} \id s.
\end{equation}
Formula \eqref{ccoeffint} and part (iii) of Lemma \ref{lemma4} together imply that $c_n(1 - a) > 0$. The fact that $c_n(1 - a)$ is a monic polynomial in $1 - a$ of degree $2n + 1$ with rational coefficients is shown in Appendix \ref{coefficients}. To estimate the remainder $R_N^{(X)}(a, w)$, we can proceed in the same manner as for $R_N^{(\varphi)}(a, z)$.
\end{proof}

\section{Discussion}\label{sec5}

In this paper, we investigated the asymptotic behaviour of the generalised trigonometric integral $ \ti(a, z, \alpha) $ and its associated modulus and phase functions for large complex values of $ z $. We derived asymptotic expansions for these functions, accompanied by explicit and computable error bounds. Special attention was given to the case where $ a $ is real and $ a < 1 $, where we showed that $ \ti(a, z, \alpha) $ has infinitely many positive real zeros. For these large zeros, we established asymptotic expansions with precise error estimates. Furthermore, we demonstrated that the asymptotic expansions for real variables are enveloping, meaning that successive partial sums of the expansions provide upper and lower bounds for the corresponding functions.

If $a > 1$, the function $\ti(a,z,\alpha)$ still has infinitely many positive real zeros. Indeed, from equation \eqref{tfg} and the asymptotic behaviour of the functions $\f(a,z)$ and $\g(a,z)$, we can conclude that
\[
z^{1 - a} \ti(a,z,\alpha) =  - \sin (z - \pi \alpha) + \mathcal{O}(z^{-1}),
\]
which shows that $\ti(a,z,\alpha)$, as a function of $z$, changes sign infinitely often for all real values of $a$. Starting with equation \eqref{tfg} and the asymptotic expansions of $\f(a,z)$ and $\g(a,z)$, and using a procedure similar to the one in \cite[\S3]{Olver1999}, we establish that the asymptotic expansion \eqref{zasymp} remains valid for the large zeros when $a > 1$ as well. However, obtaining error bounds for this expansion in this case appears to be challenging. For example, it is unclear whether the modulus function $M(a,z)$ remains non-zero for all $a > 1$ and $z > 0$, which is necessary for the phase function to be defined for all $a>1$ and $z > 0$. While numerical experiments suggest this to be the case, $M(a,z)$ can take very small values for certain combinations of $a$ and $z$, indicating that a rigorous proof might be difficult to obtain. Moreover, we found that when $\varphi(a,z)$ is well-defined for some $a > 1$ and $z > 0$, it is not necessarily injective and, therefore, not invertible for all $z > 0$.

The asymptotic expansions discussed in this paper are, in general, divergent. However, with a more sophisticated analysis, it can be shown that, for example, for large $\kappa$, the terms of the asymptotic series \eqref{zasymp} initially decrease in magnitude, reach a minimum, and then begin to diverge. Optimal truncation of the series (i.e., truncating at or near the numerically smallest term), given by $n \approx \frac{\pi}{2}\kappa$, results in an approximation whose error is exponentially small in the asymptotic variable $\kappa$. Further exponential accuracy can be achieved beyond the least term by re-expanding the remainder term $R_N^{(X)}(a,w)$ using its integral representation \eqref{Xint} (see, e.g., \cite{Berry1991}). In this context, an Euler transformation approach proves particularly effective (cf. \cite[\S3]{Boyd1990}).

As a final remark, we note that the results of this paper can be used to derive asymptotic expansions for the Fresnel integrals. Specifically, using the notation of \cite[\href{http://dlmf.nist.gov/7.2.iii}{\S7.2(iii)}]{DLMF}, we have:
\begin{align*}
\mathscr{F}(z) & = \int_z^\infty \e^{\frac{\pi}{2}\im t^2} \id t = \frac{1}{\sqrt{2\pi}} \ci\left(\tfrac{1}{2}, \tfrac{\pi}{2}z^2\right) + \frac{\im}{\sqrt{2\pi}} \si\left(\tfrac{1}{2}, \tfrac{\pi}{2}z^2\right),
\\ S(z) &= \int_0^z \sin\left(\tfrac{\pi}{2}t^2\right) \id t = \frac{1}{2} - \frac{1}{\sqrt{2\pi}} \si\left(\tfrac{1}{2}, \tfrac{\pi}{2}z^2\right),
\\ C(z) &= \int_0^z \cos\left(\tfrac{\pi}{2}t^2\right) \id t = \frac{1}{2} - \frac{1}{\sqrt{2\pi}} \ci\left(\tfrac{1}{2}, \tfrac{\pi}{2}z^2\right),
\end{align*}
for any complex $z$. However, it is clear that the asymptotic behaviour of the large (complex) zeros of $S(z)$ and $C(z)$ cannot be derived from the results presented in this paper.

\appendix

\section{Properties and computation of the coefficients}\label{coefficients}

In this appendix, we demonstrate how the coefficients $t_n(x)$ and $c_n(x)$ can be computed using explicit recursive formulae. Furthermore, we prove that both $t_n(x)$ and $c_n(x)$ are monic polynomials in $x$ of degree $2n+1$, with integer coefficients for $t_n(x)$ and rational coefficients for $c_n(x)$.

We begin by proving the formula \eqref{phidiff} for the $z$-derivative of the phase function $\varphi(a, z)$. Assume that $a < 1$ and $z > 0$. From \eqref{Mphidef}, \eqref{fdef}, and \eqref{gdef}, we can deduce that
\[
M(a,z)\e^{\im\varphi (a,z)}  =
\e^{\im z} (\g(a,z) + \im\f(a,z)).
\]
Since $a < 1$, it follows that none of the functions $M(a,z)$, $\f(a,z)$, or $\g(a,z)$ is zero. Taking the logarithmic derivative of both sides yields
\[
\cfrac{\cfrac{\d M(a,z)}{\d z}}{M(a,z)} + \im \frac{\d \varphi(a,z)}{\d z} = \im + \cfrac{\cfrac{\d \g(a,z)}{\d z} + \im \cfrac{\d \f(a,z)}{\d z}}{\g(a,z) + \im \f(a,z)}.
\]
Taking the imaginary part of each side and applying \eqref{M2} gives
\[
\frac{\d \varphi(a,z)}{\d z} = 1 + \cfrac{\g(a,z)\cfrac{\d \f(a,z)}{\d z}  - \f(a,z)\cfrac{\d \g(a,z)}{\d z} }{M^2(a,z)},
\]
which, using \eqref{fgdiff} and \eqref{M2}, simplifies to \eqref{phidiff}.

To derive a recurrence relation for the coefficients $t_n(x)$, we can use formula \eqref{phidiff}. Since $\varphi(a, z)$ extends analytically to the half-plane $\Re(z) > 0$, we can differentiate its asymptotic expansion term-by-term, yielding
\begin{equation}\label{phidasymp}
\frac{\d \varphi(a,z)}{\d z} \sim 1 + \sum_{n=0}^\infty (-1)^n \frac{t_n(1-a)}{z^{2 n + 2}}
\end{equation}
as $z \to +\infty$, with $a$ held fixed. We then rewrite \eqref{phidiff} as
\[
M^2 (a,z)\frac{\d}{\d z}\varphi (a,z) = \f(a,z)
\]
and substitute using \eqref{Masymp}, \eqref{phidasymp}, and \eqref{fasymp}. By setting $x = 1-a$, we obtain the following relation between formal series:
\[
\left( \sum_{n=0}^\infty (-1)^n \frac{(x)_{2n+1}}{n+x} \frac{1}{z^{2n}} \right) 
\left( 1 + \sum_{n=0}^\infty (-1)^n \frac{t_n(x)}{z^{2n+2}} \right) 
= \sum_{n=0}^\infty (-1)^n \frac{(x)_{2n}}{z^{2n}}.
\]
Multiplying the two series on the left-hand side and equating like powers of $z$, we obtain the following recurrence relation for the coefficients:
\[
t_0(x) = x, \quad t_n(x) = \frac{n+1}{n + x+1}(x)_{2n+2} -  \sum_{k=1}^{n} \frac{(x)_{2k + 1}}{k + x}  t_{n - k}(x)
\]
for all $n \geq 1$. Since $(x)_{2n+2}$ is divisible by $n+x+1$ and $(x)_{2k+1}$ is divisible by $k+x$, a simple induction argument shows that each $t_n(x)$ is a polynomial with integer coefficients. Moreover, it follows by induction that the leading term of $t_n(x)$ is $x^{2n+1}$. The explicit forms of the first few $t_n(x)$ are provided in Table \ref{table1}. We conjecture that each $t_n(x)$ has only positive coefficients.

\begin{table*}[!ht]
\caption{The coefficients $t_n(x)$ for $0\leq n \leq 5$.}
\begin{center}
\begin{tabular}
[c]{l l}  \hline
& \\ [-2ex]
 $n$ & $t_n(x)$ \\ [-2ex]
& \\ \hline 
 & \\ [-1.5ex]
 0 & $x$\\ [1ex]
 1 & $x^3 + 6x^2 + 6x$\\ [1ex]
 2 & $x^5 + 20x^4 + 110x^3 + 210x^2 + 120x$\\[1ex]
 3 & $x^7 + 42x^6 + 560x^5 + 3248x^4 + 8946x^3 + 11256x^2 + 5040x$\\[1ex]
 4 & $x^9 + 72x^8 + 1764x^7 + 20580x^6 + 129834x^5 + 463050x^4+ 920184x^3$\\[1ex]
& $ + 930960x^2 + 362880x$\\[1ex]
 5 & $x^{11} + 110x^{10} + 4290x^9 + 83688x^8 + 939774x^7 + 6494092x^6$ \\[1ex]
 & $+ 28332282x^5 + 77504328x^4 + 127178832x^3 + 112289760x^2 + 39916800x$\\[-1.5ex]
 & \\\hline
\end{tabular}
\end{center}
\label{table1}
\end{table*}

A recurrence scheme for calculating the coefficients $c_n(x)$ can be derived as follows. Note that the asymptotic series \eqref{Xasymp} for $X(a,w)$ is obtained by formally inverting the asymptotic series of $\Phi(a,z)$. Therefore, by a result of Fabijonas and Olver \cite[Theorem 2.4]{Olver1999}, $(-1)^{n+1} c_n(x)$ is the coefficient of $z^{-1}$ in the asymptotic expansion of $(\Phi(1-x,z))^{2n+1}$, or equivalently, the coefficient of $z^{-2n-2}$ in the asymptotic expansion of $\left(z^{-1} \Phi(1-x,z)\right)^{2n+1}$. Let us define, for each positive integer $n$, the coefficient of $z^{-2k}$ in the expansion of $\left( z^{-1} \Phi(1-x,z) \right)^{2n+1}$ as $(-1)^k d_{n,k}(x)$, that is
\[
\left( 1 - \sum_{k=0}^\infty (-1)^k \frac{t_k(x)}{2k+1} \frac{1}{z^{2k+2}} \right)^{2n+1} = \sum\limits_{k = 0}^\infty(-1)^k  \frac{d_{n,k} (x)}{z^{2k} }.
\]
With this notation, we can express $c_n(x)$ as
\[
c_n (x) = d_{n,n + 1} (x).
\]
By an exercise of Olver \cite[Chap. 1, \S8, Exer. 8.4]{Olver1997}, the coefficients $d_{n,k}(x)$ can be computed using the recurrence relation
\begin{equation}\label{drec}
d_{n,1}(x) = (2n+1) x, \quad d_{n,k}(x) = \frac{2n+1}{2k-1} t_{k-1}(x) + \frac{1}{k} \sum_{j=1}^{k-1} \frac{2j(n+1) - k}{2j-1} t_{j-1}(x) d_{n,k-j}(x)
\end{equation}
for all $n \geq 0$ and $k \geq 2$. A simple induction argument shows that each $d_{n,k}(x)$ is a polynomial with rational coefficients. Furthermore, by induction, it follows that the leading term of $d_{n,k}(x)$ is $\frac{2n+1}{2k-1} x^{2k-1}$. As a result, the leading term of $c_n(x)$ is $x^{2n+1}$. The explicit forms of the first few $c_n(x)$ are given in Table \ref{table2}. We conjecture that each $c_n(x)$ has only positive coefficients, which would follow from the recurrence \eqref{drec} if the corresponding conjecture for $t_n(x)$ holds true.

\begin{table*}[!ht]
\caption{The coefficients $c_n(x)$ for $0\leq n \leq 5$.}
\begin{center}
\begin{tabular}
[c]{l l}  \hline
& \\ [-2ex]
 $n$ & $c_n(x)$ \\ [-2ex]
& \\ \hline 
 & \\ [-1.5ex]
 0 & $x$\\ [2.75ex]
 1 & $x^3 + 9x^2 + 6x$\\ [1ex]
 2 & $x^5 + \cfrac{80}{3}x^4 + 160x^3 + 250x^2 + 120x$\\[1ex]
 3 & $x^7 + \cfrac{791}{15}x^6 + 791x^5 + 4529x^4 + 11088x^3 + 12348x^2 + 5040x$\\[1ex]
 4 & $x^9 + \cfrac{3048}{35}x^8 + \cfrac{11996}{5}x^7 + \cfrac{144924}{5}x^6 + 176016x^5 + 578466x^4 
$\\[1ex]
& $+ 1052520x^3 + 986256x^2 + 362880x$\\[1ex]
 5 & $x^{11} + \cfrac{40843}{315}x^{10} + \cfrac{356092}{63}x^9 + \cfrac{2439712}{21}x^8 + \cfrac{3907442}{3}x^7 + 8635462x^6$ \\[1ex]
 & $+ 35393952x^5 + \cfrac{271612924}{3}x^4 + 139585512x^3 + 116915040x^2 + 39916800x$\\[-1.5ex]
 & \\\hline
\end{tabular}
\end{center}
\label{table2}
\end{table*}

\section{Bounds for the basic terminant}\label{terminant}

In this appendix, we present estimates for the absolute value of the basic terminant $\Pi_p(z)$, assuming $\Re(p) > 0$. These estimates depend solely on $p$ and the argument $\arg z$ of $z$, and hence they also provide bounds for the quantity $\sup_{r \geq 1} \left| \Pi_p(zr) \right|$, which appears in Theorem \ref{thm1}. For detailed proofs, we refer the reader to \cite{Nemes2017a} and \cite{Nemes2017b}.

\begin{proposition}\label{propt1}
For any complex number $p$ with $\Re(p) > 0$, the following inequality holds:
\begin{equation}\label{eqPi}
\left| {\Pi _p (z)} \right| \le \frac{\Gamma (\Re (p))}{\left| \Gamma (p) \right|} \times \begin{cases} 1 & \text{if }\; \left|\arg z\right| \leq \frac{\pi}{4}, \\ \left|\csc ( 2\arg z)\right| & \text{if }\; \frac{\pi}{4} < \left|\arg z\right| < \frac{\pi}{2}. \end{cases}
\end{equation}
Furthermore, when both $z$ and $p$ are positive, it holds that $0 < \Pi_p(z) < 1$.
\end{proposition}

We remark that it was shown by the author in \cite{Nemes2017a} that $\left| {\Pi _p (z)} \right| \le \sqrt {\frac{\e}{4}\left( {p + \frac{3}{2}} \right)}$ provided that $\frac{\pi}{4} < \left| \arg z \right| \leq \frac{\pi}{2}$ and $p$ is real and positive. This result improves upon the bound in \eqref{eqPi} near $|\arg z| = \frac{\pi}{2}$.

\begin{proposition} For any complex number $p$ with $\Re(p) > 0$, the following inequalities hold:
\[
\left| \Pi_p (z) \right| \le \frac{1}{2}\sec ^{\Re(p)} (\arg z)\max \left(1,\e^{\Im (p)\left(  \mp \frac{\pi }{2} - \arg z \right)} \right) + \frac{1}{2}\max \left(1,\e^{\Im (p)\left(  \pm \frac{\pi }{2} - \arg z \right)} \right)
\]
and
\[
\left| \Pi_p (z) \right| \le \frac{1}{2}\sec ^{\Re(p)} (\arg z)\max \left(1,\e^{\Im (p)\left(  \mp \frac{\pi }{2} - \arg z \right)} \right) + \frac{\Gamma (\Re (p))}{2\left| \Gamma (p) \right|},
\]
for $0 \le \pm\arg z < \frac{\pi }{2}$.
\end{proposition}

The following estimate holds for positive real values of the order $p$.

\begin{proposition}  
For any $p > 0$ and $z$ with $\frac{\pi}{4} < \left| \arg z \right| < \pi$, we have  
\begin{equation}\label{eq86}  
\left| \Pi_p(z) \right| \leq \frac{\left| \csc\left( 2(\arg z - \theta) \right) \right|}{\cos^p \theta},  
\end{equation}  
where $\theta$ is the unique solution of the implicit equation  
\[
\left( p + 2 \right)\cos\left( 2\arg z - 3\theta \right) = \left( p - 2 \right)\cos\left( 2\arg z - \theta \right),
\]  
and satisfies $0 < \theta  <  - \frac{\pi}{4} + \arg z$ if $\frac{\pi}{4} < \arg z  < \frac{\pi}{2}$, $ - \frac{\pi}{2}  + \arg z  < \theta  < -\frac{\pi}{4}+\arg z$ if $\frac{\pi}{2}  \le \arg z  < \frac{3\pi}{4}$, $ - \frac{\pi}{2}  + \arg z  < \theta  < \frac{\pi }{2}$ if $\frac{3\pi}{4}  \le \arg z  < \pi$, $\frac{\pi}{4} + \arg z < \theta  <  0$ if $-\frac{\pi }{2} < \arg z  < -\frac{\pi}{4}$, $\frac{\pi}{4}  + \arg z  < \theta  < \frac{\pi}{2}+\arg z$ if $-\frac{3\pi}{4}  < \arg z  \le -\frac{\pi}{2}$ and $ - \frac{\pi}{2}  < \theta  < \frac{\pi }{2}+ \arg z$ if $-\pi < \arg z \le -\frac{3\pi}{4}$.
\end{proposition}

We remark that the value of $\theta$ in this proposition is chosen to minimise the right-hand side of inequality \eqref{eq86}.

\begin{proposition} For any complex $p$ with $\Re(p) > 0$, we have the following estimates:
\begin{align*}
\left| \Pi _p (z) \right| \le \; & \frac{1}{2}  + \frac{\left| p \right|}{2 {\Re (p)} }\Gamma \left( \frac{\Re (p)}{2} + 1 \right) \mathbf{F}\!\left( \frac{1}{2},\frac{\Re (p)}{2};\frac{\Re (p)}{2} + 1;\sin ^2 (\arg z) \right)\max \left(1,\e^{ - \Im (p)\arg z} \right) \\ & + \frac{1}{2}\max \left(1,\e^{ \Im (p)\left( { \pm\frac{\pi }{2} - \arg z} \right)} \right)
\\  \le \; & \frac{1}{2} + \frac{\left| p \right|}{2\Re (p)}\chi (\Re (p))\max \left(1,\e^{ - \Im (p)\arg z} \right) + \frac{1}{2}\max \left(1,\e^{\Im (p)\left(  \pm \frac{\pi }{2} - \arg z \right)} \right)
\end{align*}
and
\begin{align*}
\left| \Pi _p (z) \right| \le \; & \frac{1}{2}  + \frac{\left| p \right|}{2 {\Re (p)} } \Gamma \left( \frac{\Re (p)}{2} + 1 \right) \mathbf{F}\!\left( \frac{1}{2},\frac{\Re (p)}{2};\frac{\Re (p)}{2} + 1;\sin ^2 (\arg z) \right)\max \left(1,\e^{ - \Im (p)\arg z} \right) \\ & + \frac{\Gamma (\Re (p))}{2\left| \Gamma (p) \right|}
\\  \le \; & \frac{1}{2} + \frac{\left| p \right|}{2\Re (p)}\chi (\Re (p))\max \left(1,\e^{ - \Im (p)\arg z} \right) + \frac{\Gamma (\Re (p))}{2\left| \Gamma (p) \right|},
\end{align*}
for $\frac{\pi }{4} < \pm \arg z \le \frac{\pi }{2}$. Here, $\mathbf{F}$ denotes the regularised hypergeometric function \cite[Eq. \href{http://dlmf.nist.gov/15.2.E2}{(15.2.2)}]{DLMF}, and $\chi(p)$ is defined by
\[
\chi (p) = \sqrt \pi  \frac{\Gamma \left( \frac{p}{2} + 1 \right)}{\Gamma \left( \frac{p + 1}{2} \right)},\quad \Re(p) > 0.
\]
\end{proposition}

\begin{proposition} For any complex $p$ with $\Re(p) > 0$, we have the following estimates:
\begin{gather}\label{eq105}
\begin{split}
\left| \Pi _p (z) \right| & \le \e^{\Im (p)\left(  \pm \frac{\pi }{2} - \arg z \right)} \frac{\Gamma (\Re (p))}{\left| \Gamma (p) \right|}\frac{\sqrt {2\pi \Re (p)} }{2\left| \sin (\arg z) \right|^{\Re (p)} } + \left|\Pi _p (z\e^{ \mp \pi \im} )\right| \\ & \le \e^{\Im (p)\left(  \pm \frac{\pi }{2} - \arg z \right)} \frac{\Gamma (\Re (p))}{\left| \Gamma (p) \right|}\frac{\chi (\Re (p))}{\left| \sin (\arg z) \right|^{\Re (p)} } + \left|\Pi _p (z\e^{ \mp \pi \im} )\right|,
\end{split}
\end{gather}
for $\frac{\pi}{2}<\pm \arg z <\pi$.
\end{proposition}

The dependence on $|z|$ in the estimates given in \eqref{eq105} can be removed by using the bounds for $\left|\Pi_p(z \e^{\mp \pi \im})\right|$ provided earlier.

Finally, we note the following two-sided inequality, proved by Watson \cite{Watson1959}, for positive real values of $p$:
\[
\sqrt{\frac{\pi}{2} \left( p + \frac{1}{2} \right)} < \chi(p) < \sqrt{\frac{\pi}{2} \left( p + \frac{2}{\pi} \right)}.
\]
The upper bound can be used to simplify the error estimates involving $\chi(p)$.


\begin{thebibliography}{99}

\bibitem{Berry1991}
M.~V.~Berry, C.~J.~Howls, Hyperasymptotics for integrals with saddles, \emph{Proc. Roy. Soc. London Ser. A} \textbf{434} (1991), no. 1892, pp. 657--675.

\bibitem{Boyd1990}
W.~G.~C.~Boyd, Stieltjes transforms and the Stokes phenomenon, \emph{Proc. Roy. Soc. London Ser. A} \textbf{429} (1990), no. 1876, pp. 227--246.

\bibitem{Bohmer1939}
P.~E.~B\"ohmer, \emph{Differenzengleichung und bestimmte Integrale}, K. F. Koehler, Leipzig, 1939.

\bibitem{Dingle1973}
R.~B.~Dingle, \emph{Asymptotic Expansions: Their Derivation and Interpretation}, Academic Press, New York and London, 1973.

\bibitem{Erdelyi1953}
A.~Erd\'elyi, W.~Magnus, F.~Oberhettinger, F.~G.~Tricomi, \emph{Higher Transcendental Functions, Vol. II}, McGraw-Hill Book Company, Inc., New York, 1953.

\bibitem{Olver1999}
B.~R.~Fabijonas, F.~W.~J.~Olver, On the reversion of an asymptotic expansion and the zeros of the Airy functions, \emph{SIAM Rev.} \textbf{41} (1999), no. 4, pp. 762--773.

\bibitem{Iizuka2008}
K.~Iizuka, \emph{Engineering Optics}, 3rd ed., Springer, New York, 2008.

\bibitem{Kodaira2007}
K.~Kodaira, \emph{Complex Analysis}, Cambridge University Press, New York, 2007.

\bibitem{Lebedev1965}
N.~N.~Lebedev,  \emph{Special Functions and Their Applications}, Prentice-Hall Inc., Englewood Cliffs, New Jersey, 1965.

\bibitem{LSVA2016}
J.~Lobo-Segura, M.~A.~Villalobos-Arias, Zeroes of generalized Fresnel complementary integral functions, \emph{Rev. Mat.} \textbf{23} (2016), no. 2, pp. 321--338.

\bibitem{Manolakis2011}
D.~G.~Manolakis, V.~K.~Ingle, \emph{Applied Digital Signal Processing: Theory and Practice}, Cambridge University Press, Cambridge, 2011.

\bibitem{Markushevich1965}
A.~I.~Markushevich, \emph{Theory of Functions of a Complex Variable, Vol. 1}, Translated from the Russian by R.~A.~Silverman, Prentice-Hall, Englewood Cliffs, New Jersey, 1965.

\bibitem{Nemes2017a}
G.~Nemes, Error bounds for the large-argument asymptotic expansions of the Hankel and Bessel functions, \emph{Acta Appl. Math.} \textbf{150} (2017), no. 1, pp. 141--177.

\bibitem{Nemes2017b}
G.~Nemes, Error bounds for the asymptotic expansion of the Hurwitz zeta function, \emph{Proc. Roy. Soc. London Ser. A} \textbf{473} (2017), no. 2203, Article 20170363, 16 pp.

\bibitem{DLMF}
\emph{NIST Digital Library of Mathematical Functions}. \url{https://dlmf.nist.gov/}, Release 1.2.3 of 2024-12-15. F.~W.~J.~Olver, A.~B.~Olde Daalhuis, D.~W.~Lozier, B.~I.~Schneider, R.~F.~Boisvert, C.~W.~Clark, B.~R.~Miller, B.~V.~Saunders, H.~S.~Cohl, and M.~A.~McClain, eds.

\bibitem{Olver1997}
F.~W.~J.~Olver, \emph{Asymptotics and Special Functions}, AKP Classics, A K Peters Ltd., Wellesley, MA, 1997. Reprint of the 1974 original, published by Academic Press, New York.

\bibitem{Temme2015}
N.~M.~Temme, \emph{Asymptotic Methods for Integrals}, Vol. 6 of Series in Analysis, World Scientific Publishing Co. Pte. Ltd., Hackensack, New Jersey, 2015.

\bibitem{Watson1959}
G.~N.~Watson, A note on gamma functions, \emph{Edinburgh Math. Notes} \textbf{42} (1959), pp. 7--9.

\end{thebibliography}
\end{document}